\documentclass[11pt]{article}
\usepackage[margin=1in]{geometry}

\usepackage{amsmath,amsthm,amssymb,amsfonts}
\usepackage[]{algorithm,algorithmic} 
\usepackage{mathtools}
\usepackage{tcolorbox}
\usepackage{bbm}
\usepackage{todonotes}
\usepackage{tensor}
\usepackage{caption}

\newcommand{\RR}{\mathbb{R}}

\usepackage{accents}

\mathchardef\mhyphen="2D 
\DeclareMathOperator*{\argmin}{\arg\!\min}

\newtheorem{theorem}{Theorem}
\numberwithin{theorem}{section}

\newtheorem{lemma}[theorem]{Lemma}

\newtheorem{corollary}[theorem]{Corollary}

\let\originalleft\left
\let\originalright\right
\renewcommand{\left}{\mathopen{}\mathclose\bgroup\originalleft}
\renewcommand{\right}{\aftergroup\egroup\originalright}

\usepackage{titlesec} 
\titleformat{\subsubsection}[runin]{\normalfont\bfseries}{\thesubsubsection.}{3pt}{}

\begin{document}
	\title{General H\"older Smooth Convergence Rates Follow From Specialized Rates Assuming Growth Bounds}
	\author{Benjamin Grimmer\footnote{School of Operations Research and Information Engineering, Cornell University,
			Ithaca, NY 14850, USA;
			\texttt{people.orie.cornell.edu/bdg79/} \newline This material is based upon work supported by the National Science Foundation Graduate Research Fellowship under Grant No. DGE-1650441.}}
	\date{}
	\maketitle

	\begin{abstract}
		Often in the analysis of first-order methods for both smooth and nonsmooth optimization, assuming the existence of a growth/error bound or KL condition facilitates much stronger convergence analysis. Hence separate analysis is typically needed for the general case and for the growth bounded cases. We give meta-theorems for deriving general convergence rates from those assuming a growth lower bound. Applying this simple but conceptually powerful tool to the proximal point, subgradient, bundle, dual averaging, gradient descent, Frank-Wolfe, and universal accelerated methods immediately recovers their known convergence rates for general convex optimization problems from their specialized rates. New convergence results follow for bundle methods, dual averaging, and Frank-Wolfe. Our results can lift any rate based on H\"older continuous gradients and H\"older growth bounds. Moreover, our theory provides simple proofs of optimal convergence lower bounds under H\"older growth from textbook examples without growth bounds.
	\end{abstract}

	\section{Introduction}
Recent interest in first-order methods for convex optimization has grown out of their successful application in high-dimensional statistics, signal processing, and data science. The effectiveness of first-order methods in solving these large-scale applications is fundamentally tied to the problem's underlying geometry. This has led to the development of robust convergence theory, describing the algorithmic speed-ups induced by the existence of different problem structures.

Typically first-order optimization methods rely on the objective function possessing some continuity or smoothness structure. In nonsmooth optimization, an assumption like the objective function being uniformly Lipschitz continuous is usually needed. In smooth optimization, an assumption like the objective's gradient being uniformly Lipschitz is often made. In this work, we consider H\"older smooth objective functions (see~\eqref{eq:holder-smooth}), which captures both of these standard settings and facilitates the development of universal or blackbox methods as done in the recent works~\cite{Nesterov2015,Lan2015,Kerdreux2018,Diaz2021,RenegarGrimmer2021}.

Throughout the first-order optimization literature, improved convergence rates typically follow from assuming the given objective function satisfies a growth/error bound or Kurdyka-{\L}ojasiewicz condition (see~\cite{Lojasiewicz1963,Lojasiewicz1993,Kurdyka1998,Bolte2007,Bolte2017} for a sample of works developing these ideas). We consider the formalization of this through H\"older growth bounds (see~\eqref{eq:holder-growth}) with respect to some growth exponent $p$. The two most important cases of this bound are Quadratic Growth given by $p=2$ (generalizing strong convexity~\cite{Necoara2019}) and Sharp Growth given by $p=1$ (which occurs broadly in nonsmooth optimization~\cite{Burke1993}). Beyond these two special cases, nearly every objective function of interest is semialgebraic, from which the results of~\cite{Bolte2007} ensure a growth bound holds locally with some exponent (although it may be large and thus correspond to quite flat growth).
Considering the generalization given by the Kurdyka-Lojasiewicz property~\cite{Bolte2017}, prior works derived dimension dependent bounds on this exponent~\cite{Li2015} and calculus relating them to those of underlying component functions~\cite{Li2018}.

Different convergence proofs are typically employed for minimization with or without assuming the existence of a given growth lower bound. Table~\ref{tab:rates} summarizes the number of iterations required to guarantee $\epsilon$-accuracy for several well-known first-order methods (see Section~\ref{sec:analysis} for a precise description of these algorithms and the appendix for example details applying our theory).

\begin{table} 
	\begin{center}
		\begin{tabular}{@{}c | @{}c@{} | @{}c@{} | @{}c@{} | @{}c@{} | @{}c@{}}
			& General & $p>2$ & Quad.~Growth & $1<p<2$ & Sharp Growth \\
			\hline
			{\renewcommand{\arraystretch}{0.5} \begin{tabular}{@{}c@{}} Proximal Point \\ Method~\cite{Rockafellar1976,Ferris1991}\end{tabular}}  & $O\left(\dfrac{1}{\epsilon}\right)$ &  $O\left(\frac{1}{\alpha^{2/p}\epsilon^{1-2/p}}\right)$ & $O\left(\dfrac{\log\left(\Delta/\epsilon\right)}{\alpha}\right)$  &  $O\left(\frac{\log\log(\Delta/\epsilon)}{\alpha^{2/p}}\right)$ & $O\left(\dfrac{\Delta-\epsilon}{\alpha^2}\right)$ \\
			{\renewcommand{\arraystretch}{0.5} \begin{tabular}{@{}c@{}} Projected Subgrad \\ Method~\cite{Polyak1969,Polyak1979}\end{tabular}} & $O\left(\dfrac{1}{\epsilon^{2}}\right)$ &  $O\left(\frac{1}{\alpha^{2/p}\epsilon^{2-2/p}}\right)$ & $O\left(\dfrac{1}{\epsilon\alpha}\right)$ &   $O\left(\frac{1}{\alpha^{2/p}\epsilon^{2-2/p}}\right)$ & $O\left(\dfrac{\log\left(\Delta/\epsilon\right)}{\alpha^2}\right)$ \\
			{\renewcommand{\arraystretch}{0.5} \begin{tabular}{@{}c@{}} Dual Averaging\\ Variants \cite{Nesterov2009,Chen2012}\end{tabular}} & $O\left(\dfrac{1}{\epsilon^{2}}\right)$ & $O\left(\frac{1}{\alpha^{2/p}\epsilon^{2-2/p}}\right)^\star$ & $O\left(\dfrac{1}{\epsilon\alpha}\right)$ & None Known  & None Known \\
			{\renewcommand{\arraystretch}{0.5} \begin{tabular}{@{}c@{}} Proximal Bundle \\ Method~\cite{Kiwiel2000,Du2017,Diaz2021}\end{tabular}}  & $O\left(\dfrac{1}{\epsilon^{3}}\right)$ & $O\left(\frac{1}{\alpha^{2/p}\epsilon^{3-2/p}}\right)^\star$ & $O\left(\dfrac{1}{\epsilon\alpha^2}\right)$ &  $O\left(\dfrac{1}{\epsilon} + \dfrac{1}{\alpha^{4/p}\epsilon^{3-4/p}}\right)$ & $O\left(\dfrac{1}{\epsilon} + \dfrac{\Delta}{\alpha^4}\right)$ \\
			{\renewcommand{\arraystretch}{0.5} \begin{tabular}{@{}c@{}} Projected Gradient \\ Descent~\cite{Nesterov-introductory}\end{tabular}}   & $O\left(\dfrac{1}{\epsilon}\right)$ & $O\left(\frac{1}{\alpha^{2/p}\epsilon^{1-2/p}}\right)$ & $O\left(\dfrac{\log\left(\Delta/\epsilon\right)}{\alpha}\right)$ & - & -	\\
			{\renewcommand{\arraystretch}{0.5} \begin{tabular}{@{}c@{}} Frank-Wolfe \\ Variants~\cite{Frank1956,Kerdreux2018}\end{tabular}}   & $O\left(\dfrac{1}{\epsilon}\right)$ & $O\left(\frac{1}{\alpha^{2/p}\epsilon^{1-2/p}}\right)^\star$ & $O\left(\dfrac{\log\left(\Delta/\epsilon\right)}{\alpha}\right)$ & - & -	\\
			{\renewcommand{\arraystretch}{0.5} \begin{tabular}{@{}c@{}} Restarted Universal \\ Method~\cite{Roulet2020,RenegarGrimmer2021}\end{tabular}}  & $O\left(\dfrac{1}{\sqrt{\epsilon}}\right)$ & $O\left(\frac{1}{\alpha^{1/p}\epsilon^{1/2-1/p}}\right)$ & $O\left(\dfrac{\log\left(\Delta/\epsilon\right)}{\sqrt{\alpha}}\right)$ & - & - 
		\end{tabular}
	\end{center}
	\caption{Convergence rates for several methods where $\Delta=f(x_0)-f(x_*)$.  New results are denoted by $^\star$. The proximal point method makes no smoothness or continuity assumptions. The subgradient, dual averaging and bundle method rates assume Lipschitz continuity and the gradient descent, Frank-Wolfe and universal method rates assume Lipschitz gradient. Note several of these methods have been analyzed beyond these Lipschitz/smooth settings.} \label{tab:rates}
\end{table}

\paragraph{Our Contribution.}
This work presents a pair of meta-theorems for deriving general convergence rates from rates that assume the existence of a growth lower bound. In terms of Table~\ref{tab:rates}, we show that each convergence rate implies all of the convergence rates to its left: the quadratic growth column's rates imply all of the general setting's rates and each sharp growth rate implies that method's general and quadratic growth rate. More generally, our results show that any convergence rate assuming growth with exponent $p$ implies rates for any growth exponent $q>p$ and for the general setting.

Interestingly, this recovers the proximal point method's superlinear convergence when $1<p<2$ from its finite guarantee when $p=1$. Moreover, new convergence guarantees follow for Dual Averaging and Bundle Methods, which previously only had theory under quadratic growth\footnote{ \cite[Theorem 3]{Chen2012} assumed the stronger condition of strong convexity, but only quadratic growth is used in its proof, being invoked just after equation (38) therein.}. Inspired by these results, the subsequent work~\cite{Diaz2021} gave direct (and nontrivial) derivations of the bundle method rates yielded by our theory. For the Frank-Wolfe variant~\cite[Algorithm 2]{Kerdreux2018} when $p>2$, lifting their $p=2$ rate gives a faster rate of $O(1/\alpha^{2/p}\epsilon^{1-2/p})$ in terms of $\epsilon$ than the previously best known rate of $O(1/\alpha^{\frac{1}{p-1}}\epsilon^{\frac{p-2}{p-1}})$ proven by~\cite{Kerdreux2018} (see Appendix~\ref{app:Frank-Wolfe} for further discussion).

A natural question is whether the reverse implications hold (whether each column of Table~\ref{tab:rates} implies the rates to its right). If one is willing to modify the given first-order method, then the literature already provides a partial answer through restarting schemes~\cite{Nest1986,Mingrui2017,Yang2018,Roulet2020,RenegarGrimmer2021}. Such schemes repeatedly run a given first-order method until some criterion is met and then restart the method at the current iterate. 
For example, Renegar and Grimmer~\cite{RenegarGrimmer2021} show how a general convergence rate can be extended to yield a convergence rate under H\"older growth. Then together with our theory, any rate with or without a growth bound gives rates for every growth exponent and the general case.

\paragraph{Outline.} Section~\ref{sec:analysis} first formalizes our model for first-order methods and then states our rate lifting theorems and their restarting extensions. Section~\ref{sec:proofs} proves our main results via analysis based on related Fenchel conjugates. Details for the algorithms and applications shown in Table~\ref{tab:rates} are deferred to the appendix.
	\section{Rate Lifting Theorems} \label{sec:analysis}
We are primarily interested in minimizing a closed proper convex function $f \colon \RR^n \rightarrow \RR\cup\{\infty\}$ over a closed convex set $Q\subseteq \RR^n$
\begin{equation}
\min_{x\in Q} f(x)
\end{equation}
that attains its minimum value of $f_*$ on some set $X_*\subseteq Q$. 

Most first-order methods for minimizing such problems rely on $f$ being $(L,\eta)$-H\"older smooth, defined for some $L\geq 0$ and $0\leq \eta\leq 1$ as
\begin{equation}\label{eq:holder-smooth}
\|g-g'\| \leq L\|x-x'\|^{\eta} \ \text{ for all } x,x'\in Q, g\in\partial f(x), g'\in\partial f(x') \ ,
\end{equation}
where $\partial f(x) = \{g\in\RR^n \mid  f(x') \geq f(x)+\langle g, x'-x\rangle\ \forall x'\in\RR^n\}$ is the subdifferential of $f$ at $x$. When $\eta=1$, this corresponds to the common assumption of $L$-smoothness (i.e., having $\nabla f(x)$ be $L$-Lipschitz). When $\eta=0$, this captures the standard nonsmooth optimization model of having $f(x)$ itself be $L$-Lipschitz. 
The analysis of many first-order methods further assumes $(\alpha,p)$-H\"olderian growth (or a similar KL/error bound condition), defined by
\begin{equation}\label{eq:holder-growth}
f(x) \geq f_* + \alpha\ \mathrm{dist}(x,X_*)^p \ \text{ for all } x\in Q \ ,
\end{equation}
where $\mathrm{dist}(x,X_*) = \inf\{\|x-x_*\| \mid x_*\in X_*\}$.
When $p=2$, this corresponds to quadratic growth, capturing the common setting of strongly convex objectives, and when $p=1$, this corresponds to sharp growth.

Recall Table~\ref{tab:rates} shows the improved convergence rates of several first-order methods relying on H\"older smoothness and H\"older growth bounds. Namely, it considers the  Proximal Point Method with any stepsize $\rho>0$ defined by
\begin{equation}\label{eq:proximal}
x_{k+1} = \mathrm{prox}_{\rho, f}(x_{k}) := \argmin_{x\in Q}\{f(x) + \frac{1}{2\rho}\|x-x_k\|^2\} \ ,
\end{equation}
the Projected Subgradient Method defined by
\begin{equation}\label{eq:subgradient}
x_{k+1} = \mathrm{proj}_Q(x_{k}-\rho_kg_k)\text{ for some } g_k\in\partial f(x_k) \ ,
\end{equation}
where $\mathrm{proj}_Q(\cdot)$ denotes orthogonal projection onto $Q$, and Projected Gradient Descent
\begin{equation}\label{eq:gradient}
x_{k+1} = \mathrm{proj}_Q(x_{k}-\rho_k\nabla f(x_k))
\end{equation}
with stepsize $\rho_k= \|\nabla f(x_k)\|^{(1-\eta)/\eta}/L^{1/\eta}$. In the appendix, we apply our results to the existing convergence theory for these and the more sophisticated Proximal Bundle Method\footnote{Stronger convergence rates have been established for related level bundle methods~\cite{Lemarechal1995,Kiwiel1995,Lan2015}, which share many core elements with proximal bundle methods.} of~\cite{Lemarechal1975,Wolfe1975,Diaz2021}, Frank-Wolfe/Conditional Gradient Methods~\cite{Frank1956,Kerdreux2018} and a restarted variant of the Universal Gradient Method of~\cite{Nesterov2015}. Note the references above analyzed the universal settings of H\"older smoothness and H\"older growth for these three of these methods.

\subsection{Formalization of Rate Lifting Theorems}
Now we formalize our model for a generic first-order method $\mathtt{fom}$ to which our rate lifting theorems will apply. Note that for it to be meaningful to lift a convergence rate to apply to general problems, the inputs to $\mathtt{fom}$ need to be independent of the existence of a growth bound~\eqref{eq:holder-growth}, but may depend on the constraints $Q$, H\"older smoothness constants $(L,\eta)$, or optimal objective value $f_*$. 
We make the following three assumptions about $\mathtt{fom}$:
\begin{itemize}
	\item[(A1)] The method \texttt{fom} computes a sequence of iterates $\{x_{k}\}^{\infty}_{k=0}$ in $Q$. The next iterate $x_{k+1}$ is determined by the first-order oracle values $\{(f(x_j),g_j)\}_{j=0}^{k+1}$, where $g_j\in\partial f(x_j)$.
	\item[(A2)] The distance from any $x_k$ to some fixed $x_*\in\argmin f$ is at most some constant $D>0$.
	\item[(A3)] For some $p\geq 1$, there exists a function $K\colon\RR_{++}^{3}\rightarrow \RR$ such that if $f$ is $(L, \eta)$-H\"older smooth and possesses $(\alpha,p)$-H\"older growth on $B(x_*, D)\cap Q$, then for any $\epsilon>0$, {\normalfont \texttt{fom}} finds an $\epsilon$-minimizer $x_k$ with
	$$k \leq K(f(x_0)-f_*, \epsilon, \alpha) \ . $$
\end{itemize}
\paragraph{On the Generality of (A1)-(A3).}
Note that (A1) allows the computation of $x_{k+1}$ to depend on the function and subgradient values at $x_{k+1}$. Hence the proximal point method is included in our model since it is equivalent to $x_{k+1} = x_{k} - \rho_kg_{k+1},\text{ where }g_{k+1}\in\partial f(x_{k+1})$.
We remark that (A2) holds for all of the previously mentioned algorithms under reasonable selection of their stepsize parameters. In fact, many common first-order methods are nonexpansive, giving $D=\|x_0-x_*\|$. Alternatively, if the constraint set $Q$ is compact, (A2) trivially holds.
Lastly, note that the convergence bound $K(f(x_0)-f_*,\epsilon,\alpha)$ in (A3) can depend on $L,\eta,p$ even though our notation does not enumerate this. If one wants to make no continuity or smoothness assumptions about $f$, $(L,\eta)$ can be set as $(\infty,0)$ (the proximal point method is one such example as it converges independent of such structure). 

The following pair of convergence rate lifting theorems show that these assumptions suffice to give general convergence guarantees without H\"older growth and guarantees for H\"older growth with any exponent $q>p$. In terms of Table~\ref{tab:rates}, these show that each column implies the rates to its left.
\begin{theorem}\label{thm:rate-lifting}
	Consider any method $\mathtt{fom}$ satisfying (A1)-(A3) and any $(L, \eta)$-H\"older smooth function $f$. For any $\epsilon>0$, {\normalfont \texttt{fom}} will find
	$ f(x_k') - f_* \leq \epsilon $ by iteration
	$$k\leq K(f(x_0)-f_*,\epsilon,\epsilon/D^p) \ , $$
	where $x_k' \in\mathrm{argmin}_{x\in Q}\{
	f(x_k) + \langle g_k, x - x_k\rangle + \frac{L}{\eta+1}\|x-x_k\|^{\eta+1}\} $ minimizes a H\"older-smooth first-order model of $f$ over $Q$.
\end{theorem}
\begin{theorem}\label{thm:higher-order-rate-lifting}
	Consider any method $\mathtt{fom}$ satisfying (A1)-(A3) and any $(L, \eta)$-H\"older smooth function $f$ possessing $(\alpha,q)$-H\"older growth with $q>p$. For any $\epsilon>0$, {\normalfont \texttt{fom}} will find $ f(x_k') - f_* \leq \epsilon $ by iteration
	$$k\leq K(f(x_0)-f_*,\epsilon,\alpha^{p/q}\epsilon^{1-p/q}) \ , $$
	where $x_k' \in\mathrm{argmin}_{x\in Q}\{
	f(x_k) + \langle g_k, x - x_k\rangle + \frac{L}{\eta+1}\|x-x_k\|^{\eta+1}\} $ minimizes a H\"older-smooth first-order model of $f$ over $Q$.
\end{theorem}
Note these guarantees apply to the auxiliary point $x'_k$. For nonsmooth settings, when $\eta=0$, this is exactly {\normalfont \texttt{fom}}'s iterates $x_k=x_k'$.  When $\eta\in(0,1]$, this point is one projected gradient descent step away $x_k'=\mathrm{proj}_Q(x_{k}-\alpha_k\nabla f(x_k)))$ with $\alpha_k=\|x_k-x_k'\|^{1-\eta}/L$. Although $\alpha_k$ is implicitly defined based on $x_k'$, it can often be computed with minor care: When $\eta=1$, $\alpha_k=1/L$. When $Q=\mathbb{R}^n$, $\alpha_k=\|\nabla f(x_k)\|^{(1-\eta)/\eta}/L^{1/\eta}$. Otherwise, a linesearching procedure can be applied to find $\alpha_k$. One shortcoming of our lifted guarantees is the need to compute an orthogonal projection to compute $x_k'$. For methods utilizing projections, this cost is marginal, but for methods like Frank-Wolfe, additional computation is needed (perhaps just inexactly computing just the final $x_k'$).

Applying these rate lifting theorems amounts to simply substituting $\alpha$ with $\epsilon/D^p$ or $\alpha^{p/q}\epsilon^{1-p/q}$ in any guarantee depending on $(\alpha,p)$-H\"older growth. For example, the projected subgradient method using the Polyak stepsize satisfies (A2) with $D=\|x_0-x_*\|$ and converges for any $L$-Lipschitz objective with quadratic growth $p=2$ at rate
$$ K(f(x_0)-f_*,\epsilon,\alpha) = \frac{8L^2}{\alpha\epsilon} $$
establishing (A3) (a proof of this is given in the appendix for completeness). Then applying Theorem~\ref{thm:rate-lifting} recovers the method's classic convergence rate as
$$ \implies K(f(x_0)-f_*,\epsilon,\epsilon/D^2) = \frac{8L^2\|x_0-x_*\|^2}{\epsilon^2} \ . $$
For convergence rates that depend on $f(x_0)-f_*$, this simple substitution falls short of recovering the method's known rates, often off by a log term. Again taking the subgradient method as an example, under sharp growth $p=1$, convergence occurs at a rate of
$$ K(f(x_0)-f_*,\epsilon,\alpha) = \frac{4L^2}{\alpha^2}\log_2\left(\frac{f(x_0)-f_*}{\epsilon}\right)\ . $$
Then Theorem~\ref{thm:rate-lifting} ensures the following weaker general rate
$$ \implies K(f(x_0)-f_*,\epsilon,\epsilon/D) = \frac{4L^2\|x_0-x_*\|^2}{\epsilon^2}\log_2\left(\frac{f(x_0)-f_*}{\epsilon}\right) $$
and Theorem~\ref{thm:higher-order-rate-lifting} ensures the weaker quadratic growth rate
$$ \implies K(f(x_0)-f_*,\epsilon,\epsilon/D) = \frac{4L^2}{\alpha\epsilon}\log_2\left(\frac{f(x_0)-f_*}{\epsilon}\right) \ . $$
In the following subsection, we provide corollaries that remedy this issue.

\subsection{Improved Rate Lifting Theorems via Restarting}
The ideas from restarting schemes can also be applied to our rate lifting theory. We consider the following conceptual restarting method $\mathtt{restart\mbox{-}fom}$ that repeatedly halves the objective gap: Set an initial target accuracy of $\tilde{\epsilon} = 2^{N-1}\epsilon$ with  $N=\lceil\log_2((f(x_0)-f_*)/\epsilon)\rceil$. Iteratively run $\mathtt{fom}$ until an $\tilde \epsilon$-optimal solution is found, satisfying
$$f(x_k')-f_*\leq \tilde{\epsilon}$$ 
for $x_k' \in\mathrm{argmin}_{x\in Q}\{
f(x_k) + \langle g_k, x - x_k\rangle + \frac{L}{\eta+1}\|x-x_k\|^{\eta+1}\} $
and then restart $\mathtt{fom}$ at the point $x_0 \leftarrow x_k'$ with new target accuracy $\tilde{\epsilon} \leftarrow \tilde{\epsilon}/2$.

Note for the proximal point method~\eqref{eq:proximal}, subgradient method with Polyak's stepsize~\eqref{eq:subgradient}, and gradient descent~\eqref{eq:gradient}, this restarting will not change the algorithm's trajectory: this follows as (i) all three of these methods have the iterates $\{y_k\}$ produced by {\normalfont \texttt{fom}} initialized at $y_0=x_T$ satisfy $y_k=x_{k+T}$ and (ii) the proximal point method and subgradient method both have $x_k' = x_k$ and projected gradient descent has $x_k'=x_{k+1}$.
As a result, the following corollaries of our lifting theorems apply directly to these three methods without restarting.
\begin{corollary} \label{cor:rate-lifting-recurse}
	Consider any method $\mathtt{fom}$ satisfying (A1)-(A3) and any $(L, \eta)$-H\"older smooth function $f$. For any $\epsilon>0$, $\mathtt{restart\mbox{-}fom}$ will find
	$ f(x_k') - f_* \leq \epsilon $
	after at most
	$$\sum_{n=0}^{N-1}K(2^{n+1}\epsilon,2^n\epsilon, 2^n\epsilon/D^p) $$
	total iterations.
\end{corollary}
\begin{proof}
	Observe that $\mathtt{restart\mbox{-}fom}$ must have found an $\epsilon$-minimizer after the $N$th restart. Our corollary then follows from bounding the number of iterations required for each of these $N$ restarts. Run $i\in\{1\dots N\}$ of $\mathtt{fom}$ has initial objective gap at most $2^{N+1-i}\epsilon$, and so Theorem~\ref{thm:rate-lifting} ensures an $2^{N-i}\epsilon$-minimizer is found after at most
	$$ K(2^{N+1-i}\epsilon, 2^{N-i}\epsilon, 2^{N-i}\epsilon/D^p)$$
	iterations. Summing this bound over all $i$ gives the claimed result. \qed
\end{proof}
\begin{corollary} \label{cor:rate-lifting-higher-order}
	Consider any method $\mathtt{fom}$ satisfying (A1)-(A3) and any $(L, \eta)$-H\"older smooth function $f$ possessing $(\alpha,q)$-H\"older growth with $q>p$. For any $\epsilon>0$, $\mathtt{restart\mbox{-}fom}$ will find
	$ f(x_k') - f_* \leq \epsilon $ after at most
	$$\sum_{n=0}^{N-1}K(2^{n+1}\epsilon,2^n\epsilon, \alpha^{p/q}(2^n\epsilon)^{1-p/q}) $$
	total iterations.
\end{corollary}
\begin{proof}
	Along the same lines as the proof of Corollary~\ref{cor:rate-lifting-recurse}, this corollary follows from bounding the number of iterations required for each of $\mathtt{restart\mbox{-}fom}$'s $N$ restarts. Run $i\in\{1\dots N\}$ of $\mathtt{fom}$ has initial objective gap at most $2^{N+1-i}\epsilon$, and so Theorem~\ref{thm:higher-order-rate-lifting} ensures an $2^{N-i}\epsilon$-minimizer is found after at most
	$$ K(2^{N+1-i}\epsilon, 2^{N-i}\epsilon, \alpha^{p/q}(2^{N-i}\epsilon)^{1-p/q})$$
	iterations. Summing this bound over all $i$ gives the claimed result. \qed
\end{proof}

Applying these corollaries to every entry in Table~\ref{tab:rates} verifies our claim that each column implies those to its left (as well as all of the omitted columns with $p\in(1,2)\cup(2,\infty)$). For example, the proximal point method has (A2) hold with $D=\|x_0-x_*\|$ and (A3) hold with
$$ K(f(x_0)-f_*,\epsilon,\alpha) = \frac{f(x_0)-f_*-\epsilon}{\rho\alpha^2} $$
under sharp growth $p=1$ (a proof of this fact is given in the appendix for completeness).
Then Corollary~\ref{cor:rate-lifting-recurse} recovers the proximal point method's general convergence rate of
$$  \sum_{n=0}^NK(2^{n+1}\epsilon,2^n\epsilon, 2^n\epsilon/D) =\sum_{n=0}^N \frac{2^{n+1}\epsilon-2^n\epsilon}{\rho(2^n\epsilon/D)^2} = \sum_{n=0}^N \frac{D^2}{\rho2^n\epsilon} = \frac{2\|x_0-x_*\|^2}{\rho\epsilon} $$
and Corollary~\ref{cor:rate-lifting-higher-order} recovers its linear convergence rate under quadratic growth $q=2$ of
\begin{align*}
\sum_{n=0}^NK(2^{n+1}\epsilon,2^n\epsilon, \alpha^{1/2}(2^n\epsilon)^{1/2}) &= \sum_{n=0}^N \frac{2^{n+1}\epsilon-2^{n}\epsilon}{\rho(2^{n/2}\alpha^{1/2}\epsilon^{1/2}))^2} = \sum_{n=0}^N \frac{1}{\rho\alpha}=\frac{N}{\rho\alpha} \ .
\end{align*}
Likewise, the $O(\sqrt{L/\alpha}\log(f(x_0)-x(x_*))/\epsilon)$ rate that follows from restarting an accelerated method for $L$-smooth optimization under quadratic growth $p=2$ recovers Nesterov's classic accelerated convergence rate as
$$  \sum_{n=0}^NK(2^{n+1}\epsilon,2^n\epsilon, 2^n\epsilon/D^2) =\sum_{n=0}^N O\left(\sqrt{\frac{LD^2}{2^{n}\epsilon}}\log(2)\right) = O\left(\sqrt{\frac{LD^2}{\epsilon}}\right)\ . $$

\subsection{Recovering Lower Bounds on Oracle Complexity}
The contrapositive of our rate lifting theorems immediately lifts complexity lower bounds from the general case to apply to the specialized case of H\"older growth. Well-known simple examples~\cite{Nesterov-introductory} give complexity lower bounds when no growth bound is assumed for first-order methods satisfying
\begin{itemize}
	\item[(A4)] For every $k\geq 0$, $x_{k+1}$ lies in the span of $\{g_i\}_{i=0}^k$.
\end{itemize}
Optimal lower bounds under H\"older growth were claimed by Nemirovski and Nesterov~\cite[page 26]{Nest1986}, although no proof is given. Here we note that the simple examples from the general case combined with our theory suffice to give concise derivations of lower bounds on the convergence rates under growth conditions.

For $M$-Lipschitz nonsmooth optimization, a simple example shows any method satisfying (A4) cannot guarantee finding an $\epsilon$-minimizer in fewer than $M^2\|x_0-x_*\|^2/16\epsilon^2$ subgradient evaluations. Consequently, applying Theorem~\ref{thm:rate-lifting}, any method satisfying (A1)-(A4) for $M$-Lipschitz optimization with $(\alpha,p)$-H\"older growth must have its rate $K(f(x_0)-f_*, \epsilon, \alpha)$ bounded below by
$$ K(f(x_0)-f_*, \epsilon, \epsilon/D^p) \geq \frac{M^2\|x_0-x_*\|^2}{16\epsilon^2}\ .$$
For example, this ensures the Polyak subgradient method's $O(1/\alpha\epsilon)$ rate under quadratic growth cannot be improved by more than small constants. The exponent on $\alpha$ cannot decrease since that would beat the general cases' lower bound dependence on $\|x_0-x_*\|^2$ (as $D=\|x_0-x_*\|$ here) and similarly the exponent of $\epsilon$ cannot be improved due to the lower bound dependence of $1/\epsilon^2$.

Likewise, for $L$-smooth optimization, a simple example shows no method satisfying (A4) can guarantee computing an $\epsilon$-minimizer using fewer than $\sqrt{3L\|x_0-x_*\|^2/32\epsilon}$ iterations. Applying Theorem~\ref{thm:rate-lifting}, we conclude any method satisfying (A1)-(A4) for smooth optimization with $(\alpha,p)$-H\"older growth has its rate $K(f(x_0)-f_*, \epsilon, \alpha)$ bounded below by
$$ K(f(x_0)-f_*, \epsilon, \epsilon/D^p) \geq \sqrt{\frac{3L\|x_0-x_*\|^2}{32\epsilon}}\ .$$
One can immediately conclude the optimal convergence smooth strongly convex functions grows with at least $O(1/\sqrt{\alpha})$ (otherwise lifting the rate would violate the general lower bounds dependence of $O(\|x_0-x_*\|)$). 


\section{Proofs of the Rate Lifting Theorems~\ref{thm:rate-lifting} and~\ref{thm:higher-order-rate-lifting}} \label{sec:proofs}
The proofs of our two main results (Theorems~\ref{thm:rate-lifting} and~\ref{thm:higher-order-rate-lifting}) rely on several properties of Fenchel conjugates, which we will first review. The Fenchel conjugate of a function $f\colon \RR^n\rightarrow \RR\cup\{\infty\}$ is
$$ f^*(g) := \sup_x\{ \langle g,x\rangle - f(x)\} \ .$$
We say that $f\geq h$ if for all $x\in\RR^n$, $f(x)\geq h(x)$. The conjugate reverses this partial ordering, having
$ f \geq h \implies h^* \geq f^*$.
Applying the conjugate twice $f^{**}$ gives the largest closed convex function majorized by $f$ (that is, the ``convex envelope'' of $f$), and consequently, for closed convex functions, $f^{**}=f$.

The $(L,q)$-H\"older smoothness condition~\eqref{eq:holder-smooth} is equivalent to having upper bounds of the following form hold for every $x\in\RR^n$ and $g\in\partial f(x)$
\begin{equation} \label{eq:smoothness-primal-upperbound}
f(x') \leq f(x)+\langle g, x'-x \rangle +\frac{L}{\eta+1}\|x'-x\|^{\eta+1} \ .
\end{equation}
Taking the conjugate of this convex upper bound gives an equivalent dual condition: a function $f$ is $(L,\eta)$-H\"older smooth if and only if its Fenchel conjugate has the following lower bound for every $g\in\RR^n$ and $x\in\partial f^*(g)$
\begin{equation} \label{eq:smoothness-dual-lowerbound}
f^*(g') \geq f^*(g) +\langle x, g'-g\rangle + \begin{cases}
\frac{\eta}{(\eta+1)L^{1/\eta}}\|g'-g\|^{(\eta+1)/\eta} & \text{if } \eta>0\\
\delta_{\|g'-g\|\leq L}(g') & \text{if } \eta=0 \ ,	
\end{cases}
\end{equation}
where $\delta_{\|g'-g\|\leq L}(g')=\begin{cases} 0 & \text{if } \|g'-g\|\leq L \\ \infty & \text{otherwise}\end{cases}$ is an indicator function. This corresponds to the Fenchel conjugate being sufficiently uniformly convex (in particular when $\eta=1$, this dual condition is exactly strong convexity).

\subsection{Proof of Theorem~\ref{thm:rate-lifting}}\label{subsec:general}
Suppose that no iteration $k\leq T$ has $x_k'$ as an $\epsilon$-minimizer of $f$ over $Q$. 
We consider the following convex auxiliary functions: at each $x_k$, the $(L,\eta)$-H\"older smoothness and subgradient $g(x_k)= g_k\in\partial f(x_k)$ define
$$ h_{x_k}(x) := f(x_k)+\langle g(x_k), x-x_k \rangle +\frac{L}{\eta+1}\|x-x_{k}\|^{\eta+1}$$
and at each minimizer $\bar x \in X_*$ with subgradient $g(\bar x)\in\partial f(\bar x)$ normal to the constraint set $-g(\bar x)\in N_{Q}(\bar x)$ (certifying the optimality of $\bar x$) define
$$ h_{\bar x}(x) := f_*+\langle g(\bar x), x-\bar x \rangle+ \frac{L}{\eta+1}\|x-\bar x\|^{\eta+1} \ . $$	
Then we focus on the convex envelope surrounding these models
\begin{equation} \label{eq:auxillary-obj}
h(x) := \left(\min\left\{h_{\bar x}(\cdot)\mid \bar x\in \{x_0,\dots, x_T\}\cup X_*\right\}\right)^{**}(x)\ .
\end{equation}
We consider the auxiliary minimization problem of 
$ \min_{x\in Q} h(x) $,
which shares and improves on the structure of $f$ as described in the following three lemmas.
\begin{lemma}\label{lem:agreement}
	The objectives $f$ and $h$ have $f(\bar x)=h(\bar x)$ and $g(\bar x)\in\partial h(\bar x)$ for each $\bar x\in \{x_0,\dots, x_T\}\cup X_*$.
\end{lemma}
\begin{lemma}\label{lem:smoothness}
	The objective $h$ is $(L,\eta)$-H\"older smooth.
\end{lemma}
\begin{lemma}\label{lem:growth}
	The objective $h$ has $(\epsilon/D^p,p)$-H\"older growth on $B(x_*, D)\cap Q$.
\end{lemma}

Before proving these three results, we show that they suffice to complete our proof. Lemma~\ref{lem:agreement} and assumption (A1) together ensure that applying \texttt{fom} to either $f$ or $h$ produces the same sequence of iterates up to iteration $T$. Since $f$ and $h$ both minimize on $X_*$ (as each $\bar x\in X_*$ has a subgradient $g(\bar x)\in -N_Q(\bar x)$ for both $f$ and $h$), no $x_k$ with $k\leq T$ is an $\epsilon$-minimizer of $h$. Hence applying the structural conditions from Lemmas~\ref{lem:smoothness} and~\ref{lem:growth} with assumption (A3) on $h$ completes our rate lifting argument as
$$T < K(h(x_0)-h(x_*),\epsilon,\epsilon/D^p)=K(f(x_0)-f_*,\epsilon,\epsilon/D^p) \ . $$

\subsubsection{Proof of Lemma~\ref{lem:agreement}}
By definition, each $g(x_i)$ provides a linear lower bound on $f$ as
$$ f(x) \geq f(x_i) + \langle g(x_i), x-x_i\rangle \ . $$
Noting that all $\bar x\in \{x_0,\dots, x_T\}\cup X_*$ have $h_{x} \geq f$, it follows that for all $x\in Q$
$$\min\{h_{\bar x}(x) \mid \bar x\in \{x_0,\dots, x_T\}\cup X_*\}\geq f(\bar x) + \langle g(\bar x), x-\bar x\rangle \ . $$
Since this linear lower bound is convex, it is also a lower bound on the convex envelope $h$. Setting $x=\bar x$, equality holds with this linear lower bound. Thus $f(\bar x) = h(\bar x)$ and $g(\bar x)$ is also a subgradient of $h$ at $\bar x$.

\subsubsection{Proof of Lemma~\ref{lem:smoothness}}
Here our proof relies on the dual perspective of $(L,\eta)$-H\"older smoothness given by~\eqref{eq:smoothness-dual-lowerbound}. 
Since each $h_{\bar x}$ is $(L,\eta)$-H\"older smooth on $\mathbb{R}^n$, each $h^*_{\bar x}$ satisfies the dual lower bounding condition~\eqref{eq:smoothness-dual-lowerbound}. Observe that 
\begin{align*}
h^*(y) &= \left(\min_{\bar x\in \{x_0,\dots, x_T\}\cup X_*}\left\{h_{\bar x}(\cdot)\right\}\right)^{*}(y)\\
&= \sup_x\left\{\langle y,x\rangle - \min_{\bar x\in \{x_0,\dots, x_T\}\cup X_*}\left\{h_{\bar x}(x)\right\}\right\}\\
&= \max_{\bar x\in \{x_0,\dots, x_T\}\cup X_*}\left\{h^*_{\bar x}(y)\right\} \ .
\end{align*}
Consequently, $h^*$ also satisfies the dual condition~\eqref{eq:smoothness-dual-lowerbound} since it is the maximum of functions satisfying this lower bound. Therefore $h^{**}=h$ retains the $(L,\eta)$-H\"older smoothness of each of the models $h_{\bar x}$ and the original objective $f$.

\subsubsection{Proof of Lemma~\ref{lem:growth}}
For any $x\in B(x_*,D)\in Q$ and $k\in\{0,\dots,T\}$, the function $h_{x_k}$ lies above our claimed growth bound as
$$h_{x_k}(x) \geq h_{x_k}(x_k') \geq f_*+\epsilon \geq f_*+\frac{\epsilon}{D^p}\|x-x_*\|^p \ , $$
where the first inequality uses that $x_k'$ minimizes $h_{\bar x}$ over $Q$, the second uses $h_{\bar x}\geq f$ and no $\epsilon$-minimizer has been found, and the last inequality uses that $x\in B(x_*,D)$. Similarly, our growth bound holds for any $h_{\bar x}$ with $\bar x\in X^*$ as 
\begin{align*}
h_{\bar x}(x)&=f_*+\langle g(\bar x), x-\bar x\rangle + \frac{L}{\eta+1}\|x-\bar x\|^{\eta+1}\\
&\geq f_* + \frac{L}{\eta+1}\|x-\bar x\|^{\eta+1}\\
&\geq f_*+\frac{L}{(\eta+1)D^{p-(\eta+1)}}\mathrm{dist}(x,X_*)^p\\
&\geq f_*+\frac{\epsilon}{D^p}\mathrm{dist}(x,X_*)^p \ , 
\end{align*}
where the first inequality uses the optimality condition that $-g(\bar x)$ is normal to $Q$ and $x\in Q$, the second inequality uses that $x\in B(x_*,D)$, and the last inequality uses that $\frac{LD^{\eta+1}}{\eta+1} \geq f(x_0)-f_* \geq \epsilon$.
Hence the minimum $\min_{\bar x\in \{x_0,\dots, x_T\}\cup X_*}\left\{h_{\bar x}(x)\right\}$ satisfies our claimed H\"older growth lower bound. Since this lower bound is convex, its convex envelope $h$ must also satisfy the H\"older bound~\eqref{eq:holder-growth}.

\subsection{Proof of Theorem~\ref{thm:higher-order-rate-lifting}} \label{subsec:higher-order}
Suppose that no iteration $k\leq T$ has $x_k'$ as an $\epsilon$-minimizer of $f$ over $Q$. Consider the auxiliary objective $h(x)$ defined by~\eqref{eq:auxillary-obj}. Then Lemmas~\ref{lem:agreement} and~\ref{lem:smoothness} show $h$ agrees with $f$ everywhere $\texttt{fom}$ visits and is $(L,\eta)$-H\"older smooth. From this, (A1) ensures that applying \texttt{fom} to either $f$ or $h$ produces the same sequence of iterates up to iteration $T$. Since $f$ and $h$ both minimize on $X_*$, no $x_k$ with $k\leq T$ is an $\epsilon$-minimizer of $h$.
As before, $h$ further satisfies a H\"older growth lower bound.
\begin{lemma}\label{lem:growth-p-q}
	The objective $h$ has $(\alpha^{p/q}\epsilon^{1-p/q},p)$-H\"older growth on $B(x_*, D)$.
\end{lemma}
\noindent Hence $h$ is H\"older smooth and has H\"older growth with exponent $p$. Then (A3) ensures
\begin{equation*} T < K(h(x_0)-h(x_*),\epsilon,\alpha^{p/q}\epsilon^{1-p/q})=K(f(f(x_0)-f_*,\epsilon,\alpha^{p/q}\epsilon^{1-p/q}) \ . 
\end{equation*}

\subsubsection{Proof of Lemma~\ref{lem:growth-p-q}}
This proof follows the same general approach used in proving Lemma~\ref{lem:growth}. For any $x\in B(x_*,D)$ and $\bar x\in \{x_0,\dots, x_T\}\cup X_*$, $h_{\bar x}$ lies above our claimed growth bound: Any $x$ with $\|x-x_*\| >(\epsilon/\alpha)^{1/q} $ has
$$h_{\bar x}(x) \geq f(x) \geq f_*+\alpha\|x-x_*\|^q \geq  f_*+\alpha^{p/q}\epsilon^{1-p/q}\|x-x_*\|^p \ . $$
Any $x$ with $\|x-x_*\| \leq (\epsilon/\alpha)^{1/q}$ and $\bar x=x_k$ for some $k$ has
$$h_{x_k}(x) \geq h_{x_k}(x_k') \geq f_* +\epsilon \geq f_* +\alpha^{p/q}\epsilon^{1-p/q}\|x-x_*\|^p $$
and any $\bar x\in X_*$ has 
\begin{align*}
h_{\bar x}(x)&=f_*+\langle g(\bar x), x-\bar x\rangle + \frac{L}{\eta+1}\|x-\bar x\|^{\eta+1}
\geq f_*+\frac{\epsilon}{(\epsilon/\alpha)^{p/q}}\mathrm{dist}(x,X_*)^p \ . 
\end{align*}
Hence $\min_{k\in\{0,\dots, T, *\}}\left\{h_k(x)\right\}$ satisfies our claimed H\"older growth lower bound. Since this lower bound is convex, the convex envelope $h$ must also satisfy the H\"older bound~\eqref{eq:holder-growth}.

	\appendix
	\section{Detailed Application to the Proximal Point Method}
First, we verify (A2) holds for the proximal point method.
\begin{lemma}\label{lem:proximal-distance}
	For any minimizer $x_*$, (A2) holds with $D=\|x_0-x_*\|$.
\end{lemma}
\begin{proof}
	The proximal operator $\mathrm{prox}_{\rho, f}(\cdot)$ is nonexpansive~\cite{Parikh2014}. Then since any minimizer $x_*$ of $f$ is a fixed point of $\mathrm{prox}_{\rho, f}(\cdot)$, the distance from each iterate to $x_*$ must be nonincreasing.
\end{proof}
Assuming sharpness (H\"older growth with $p=1$) facilitates a finite termination bound on the number of iterations before an exact minimizer is found (see~\cite{Ferris1991} for a more general proof and discussion of this finite result). Below we compute the resulting function $K(f(x_0)-f(x_*),\alpha,\epsilon)$ that satisfies (A3).
\begin{lemma}\label{lem:proximal-sharp}
	Consider any convex $f$ satisfying~\eqref{eq:holder-growth} with $p=1$. Then for any $\epsilon>0$, the proximal point method with stepsize $\rho>0$ will find an $\epsilon$-minimizer $x_k$ with 
	$$ k\leq K(f(x_0)-f(x_*),\epsilon,\alpha) = \frac{f(x_0)-f(x_*) - \epsilon}{\rho\alpha^2} \ .$$
\end{lemma}
\begin{proof}
	The optimality condition of the proximal subproblem is $(x_{k+1}-x_{k})/\rho \in \partial f(x_{k+1})$. Hence convexity ensures
	$$ \|x_{k+1}-x_{k}|\|x_{k+1}-x_*\|/\rho \geq \langle (x_{k+1}-x_{k})/\rho, x_{k+1} - x_*\rangle \geq f(x_{k+1}) - f(x_*).$$
	Then supposing that $x_{k+1}\neq x_*$, the sharp growth bound ensures
	$$\|x_{k+1}-x_{k}\| \geq \rho (f(x_{k+1}) - f(x_*))/\|x_{k+1}-x_*\| \geq \rho\alpha \ .$$
	Noting the proximal subproblem is $\rho$-strongly convex, the objective has
	$$f(x_{k+1}) \leq f(x_{k}) - \frac{\|x_{k+1}-x_{k}\|^2}{\rho} \leq f(x_{k}) - \rho\alpha^2 \ , $$
	giving a constant decrease at each iteration until a minimizer is found.
\end{proof}
Applying Corollary~\ref{cor:rate-lifting-recurse} and~\ref{cor:rate-lifting-higher-order}, we recover every other setting's convergence rates from the $p=1$ rate. For $q>1$-growth, we reach $\epsilon$-accuracy after
$$\sum_{n=0}^{N-1}K(2^{n+1}\epsilon,2^n\epsilon, \alpha^{1/q}(2^n\epsilon)^{1-1/q}) = \sum_{n=0}^{N-1} \frac{(2^n\epsilon)^{2/q-1}}{\rho\alpha^{2/q}} \ .$$
This recovers the three major known regimes of the proximal point methods behavior. For $1<q<2$, this is converging superlinearly, for $q=2$ converging linearly, and for $q>2$, converging sublinearly.

\section{Detailed Application to the Polyak Subgradient Method}
Much like the proximal point method, the distance from each iterate of the subgradient method to a minimizer is nonincreasing when a sufficiently careful stepsize is employed. Below we consider the Polyak stepsize.
\begin{lemma}\label{lem:subgradient-distance}
	For any minimizer $x_*$, (A2) holds with $D=\|x_0-x_*\|$.
\end{lemma}
\begin{proof}
	The convergence of the projected subgradient method is governed by the following
	\begin{align}
	\|x_{k+1} - x_*\|^2  &\leq \|x_{k} - \rho_kg_k -x_*\|^2\nonumber\\
	&= \|x_{k}-x_*\|^2 -2\langle\rho_kg_k, x_{k}-x_*\rangle + \rho_k^2\|g_k\|^2\nonumber\\
	& \leq \|x_{k}-x_*\|^2 -2\rho_k (f(x_k)-f(x_*)) + \rho_k^2\|g_k\|^2\nonumber\\
	& = \|x_{k}-x_*\|^2 -\frac{(f(x_k)-f(x_*))^2}{\|g_k\|^2}\nonumber\\
	& \leq \|x_{k}-x_*\|^2 -\frac{(f(x_k)-f(x_*))^2}{L^2} \ , \label{eq:subgradient-recurrence}
	\end{align}
	where the first inequality uses the convexity of $f$ and the second uses our assumed subgradient bound.
	Thus the distance from each iterate to $x_*$ is nonincreasing.
\end{proof}
As an alternative stepsize example, consider the step size $\rho_k = 2\epsilon/\|g_k\|^2$. Then the distance to optimal strictly decreases every step until an $\epsilon$-minimizer is found as
\begin{align}
\|x_{k+1} - x_*\|^2 &= \|x_{k}-x_*\|^2 -2\langle\rho_kg_k, x_{k}-x_*\rangle + \rho_k^2\|g_k\|^2\nonumber\\
& \leq \|x_{k}-x_*\|^2 - 2\epsilon\frac{f(x_k)-f(x_*) - \epsilon}{L^2} \ .\nonumber
\end{align}
Below we give a simple proof that the projected subgradient method under $L$-Lipschitz continuity and quadratic growth finds an $\epsilon$-minimizer within $O(L^2/\alpha\epsilon)$ iterations.
\begin{lemma}\label{lem:subgradient-quadratic}
	Consider any convex $f$ satisfying~\eqref{eq:holder-growth} with $p=2$. Then for any $\epsilon>0$, the subgradient method with the Polyak stepsize will find an $\epsilon$-minimizer $x_k$ with 
	$$ k\leq K(f(x_0)-f(x_*),\epsilon,\alpha) = \frac{8L^2}{\alpha\epsilon} \ .$$
\end{lemma}
\begin{proof}
	This convergence rate follows by noting~\eqref{eq:subgradient-recurrence} implies the objective gap will halve $f(x_k)-f(x_*) \leq (f(x_0)-f(x_*))/2$ after at most
	$$ \frac{4L^2\|x_0-x_*\|^2}{(f(x_0)-f(x_*))^2} \leq \frac{4L^2}{\alpha (f(x_0)-f(x_*))}$$
	iterations. Iterating this argument, a $2^{-N}(f(x_0)-f(x_*))$-minimizer is found with
	$$ k \leq \sum_{n=0}^{N-1} \frac{4L^2}{2^{-n}\alpha(f(x_0)-f(x_*))} \leq \frac{8L^2}{2^{-N}\alpha(f(x_0)-f(x_*))} \ .$$
	Considering $N=\lceil\log_2((f(x_0)-f(x_*))/\epsilon)\rceil$, gives our claimed bound on the number of iterations needed to find an $\epsilon$-minimizer.
\end{proof}
Rather than relying on quadratic growth, assuming sharpness (H\"older growth with $p=1$) allows us to derive a linear convergence guarantee. Below we compute the resulting function $K(f(x_0)-f(x_*),\alpha,\epsilon)$ that satisfies (A3).
\begin{lemma}\label{lem:subgradient-sharp}
	Consider any convex $f$ satisfying~\eqref{eq:holder-growth} with $p=1$. Then for any $\epsilon>0$, the subgradient method with the Polyak stepsize will find an $\epsilon$-minimizer $x_k$ with 
	$$ k\leq K(f(x_0)-f(x_*),\epsilon,\alpha) = \frac{4L^2}{\alpha^2}\log_2\left(\frac{f(x_0)-f(x_*)}{\epsilon}\right) \ .$$
\end{lemma}
\begin{proof}
	Again, this convergence rate follows by noting~\eqref{eq:subgradient-recurrence} implies the objective gap will halve $f(x_k)-f(x_*) \leq (f(x_0)-f(x_*))/2$ after at most
	$$ \frac{4L^2\|x_0-x_*\|^2}{(f(x_0)-f(x_*))^2} \leq \frac{4L^2}{\alpha^2} $$
	iterations, which immediately establishes the claimed rate.
\end{proof}

\section{Application to Dual Averaging}
Dual averaging methods following in the line of the primal-dual methods proposed by Nesterov~\cite{Nesterov2009} offer to improve on gradient methods like~\eqref{eq:subgradient} and~\eqref{eq:gradient} by using an aggregate of the first-order models seen. Chen et al.~\cite{Chen2012} propose an Optimal Regularized Dual Averaging Method (and a multistage variant) attaining the optimal rate for both $L$-smooth and nonsmooth $M$-Lipschitz problems with or without a quadratic growth bound. Formally they assume strong convexity rather than quadratic growth, but the proof of Theorem 3 therein only relies on quadratic growth (invoked shortly after equation (38)). Their theorem shows optimal convergence under $\alpha$-quadratic growth at a rate of 
$$ K(\Delta, \epsilon, \alpha) = f\sqrt{\frac{L}{\alpha}}\log(\Delta/\epsilon) + \frac{1024 M^2}{\alpha\epsilon} \ . $$
Our theory shows their dual averaging method is in fact optimal under the wider setting of H\"older growth. Applying our Corollary~\ref{cor:rate-lifting-higher-order}, we see that their dual averaging scheme (which already restarts once the objective gap is halved) is optimal both for smooth and nonsmooth problems under H\"older growth: When $L>0$ but $M=0$
$$\sum_{n=0}^{N-1}K(2^{n+1}\epsilon, 2^n\epsilon, \alpha^{2/q}(2^n\epsilon)^{1-2/q}) = \sum_{n=0}^{N-1} 4\sqrt{\frac{L}{\alpha^{2/q}(2^n\epsilon)^{1-2/q}}}\log(2) = O\left(\frac{L^{1/2}}{\alpha^{1/q}\epsilon^{1/2-1/q}}\right)$$
and when $L=0$ and $M>0$
$$\sum_{n=0}^{N-1}K(2^{n+1}\epsilon, 2^n\epsilon, \alpha^{2/q}(2^n\epsilon)^{1-2/q}) = \sum_{n=0}^{N-1} \frac{1024M^2}{\alpha^{2/q}(2^n\epsilon)^{2-2/q}} = O\left(\frac{M^2}{\alpha^{2/q}\epsilon^{2-2/q}}\right) \ .$$

\section{Application to the Proximal Bundle Method}
Bundle methods have a long history throughout the nonsmooth optimization literature. These methods build sufficiently accurate models of an objective to ensure descent, despite its nonsmoothness. Until recently, convergence rate theory was lacking in terms of the total number of subgradient evaluations required. Kiwiel~\cite{Kiwiel2000} showed a $O(1/\epsilon^{-3})$ convergence rate in 2000 and then in 2017, Du and Ruszczy\'{n}ski~\cite{Du2017} showed a $O(1/\epsilon\alpha^2)$ rate under quadratic growth, up to logarithmic terms. Both rates fall short of the optimal $O(1/\epsilon^2)$ and $O(1/\epsilon\alpha)$. 

The boundness condition~(A2) is verified by~\cite[(7.64)]{Ruszczynski2006}. Then applying our theory shows these rates are similarly suboptimal as Theorem~\ref{thm:rate-lifting} shows Du and Ruszczy\'{n}ski's $O(1/\epsilon\alpha^2)$ rate implies Kiwiel's $O(1/\epsilon^3)$ rate and a restarting argument shows the converse. Applying our Theorem~\ref{thm:higher-order-rate-lifting} gives new rates for the bundle method in every $2<p<\infty$ growth setting of $O(1/\alpha^{2/p}\epsilon^{3-2 /p})$. Subsequent to this work, in~\cite{Diaz2021}, the author directly derived these bounds (as well as bounds using stepsize rules fixing the suboptimality). Note direct proofs of these rates ended up being nontrivial.

\section{Application to the Frank-Wolfe Method} \label{app:Frank-Wolfe}
Frank-Wolfe~\cite{Frank1956} repeatedly solves linear optimization subproblems (based on gradient evaluations) over a given compact constraint set $Q$. A bound $D$ on the diameter of this compact constraint set suffices for (A2). Recently, a variant of Frank-Wolfe including (fractional) away-steps and periodic restarting was analyzed under various H\"older smoothness conditions and any $p\geq 2$-H\"older growth bound in~\cite[Theorem 6.2]{Kerdreux2018}. They assume the constraint set $Q$ satisfies a certain $\delta$-scaling property (Definition 3.3 therein, which is satisfied, for example, by all polytopes) and utilize curvature bounds (away and standard) on $f$, which are both bounded by the smoothness constant $L$ when $\eta=1$.
	
In terms of the notation and numbering of~\cite{Kerdreux2018}, their equation (15) is missing a power of $r$ in the $\mu$ term in both of the last two inequalities (see Lemma 3.6, which is being invoked there and requires $\mu=c/\delta$ to be raised to the power of $r=1/(1-\theta)$). This missing exponent can be carried forward through their arguments directly, only modifying their final convergence rates by including a power of $r$ in $\mu$.
Converting their quantities to match the notations of this paper, the corrected convergence rates from their Theorem 5.1 are $O(1/\alpha^{\frac{1}{p-1}}\epsilon^{\frac{p-2}{p-1}})$ when $p>2$ and $O(\log(1/\epsilon)/\alpha)$ when $p=2$. Applying our Corollary~\ref{cor:rate-lifting-higher-order} with their quadratic growth guarantee to the $p>2$ case proves an improved dependence on $\epsilon$ of $O(1/\alpha^{2/p}\epsilon^{1-2/p})$. An important weakness of this improved rate to note is that it holds on the auxiliary sequence $x_k'$, which requires orthogonal projections to compute from $x_k$, not typically used by Frank-Wolfe.

\section{Application to the Universal Fast Method with Restarting}
For any $(L,\eta)$-H\"older smooth, convex optimization problem, Nesterov's Universal Fast Gradient Method~\cite{Nesterov2015} attains the optimal convergence rate and the restarting scheme of~\cite{RenegarGrimmer2021} improved this convergence theory under any $(\alpha,p)$-H\"older growth. Let us briefly verify the claims in Table~\ref{tab:rates} for any $L$-smooth function. Nesterov's Theorem 3 shows a rate of $O(\sqrt{LD^{2}/\epsilon})$ for any bound $D>0$ on the iterates which Renegar and Grimmer's Corollary 10 shows (via restarting) implies faster rates of $O(\sqrt{L/\alpha}\log(\Delta/\epsilon))$ if $p=2$ and $O(\sqrt{L}/\alpha^{1/p}\epsilon^{1/2-1/p})$ otherwise. Applying our theory provides the reverse implication: the $p=2$ setting implies rates for $p>2$ and the general setting.

	\bibliographystyle{plain}
	{\small \bibliography{references}}

\end{document}